\documentclass[10pt]{amsart}
\usepackage{ amssymb }
\usepackage[margin=0.9in]{geometry}
\usepackage{ amsmath}
\DeclareSymbolFont{largesymbolsA}{U}{txexa}{m}{n}
\DeclareMathSymbol{\varprod}{\mathop}{largesymbolsA}{16}
\newtheorem{theorem}{Theorem}[section]

\newtheorem{lem}[theorem]{Lemma}
\newtheorem{prop}[theorem]{Proposition}
\newtheorem{cor}{Corollary}[theorem]
\theoremstyle{definition}

\theoremstyle{remark}

\DeclareMathOperator{\sign}{sgn}
\numberwithin{equation}{section}
\newcommand{\C}{\mathfrak{C}}

\begin{document}
\author{Ankush Goswami}
\address{Research Institute for Symbolic Computation (RISC), JKU, Linz.}
%\curraddr{}
\email{ankushgoswami3@gmail.com, ankush.goswami@risc.jku.at}
\title[Some formulae for coefficients of restricted $q$-products]{Some formulae for coefficients in restricted $q$-products}
 \author{Venkata Raghu Tej Pantangi}
 \address{Department of Mathematics, Southern University of Science and Technology (SUSTECH), Shenzhen, China.}
% %\curraddr{}
\email{pvrt1990@gmail.com, pantangi@sustech.edu.cn}

\thanks{}

    %author two information

\date{}

%\dedicatory{}
\begin{abstract}
In this paper, we derive some formulae involving coefficients of polynomials which occur quite naturally in the study of restricted partitions. Our method involves a recently discovered sieve technique by Li and Wan (Sci. China. Math. 2010). Based on this method, by considering cyclic groups of different orders we obtain some new results for these coefficients. The general result (see Theorem \ref{main00}) holds for any group of the form $\mathbb{Z}_{N}$ where $N\in\mathbb{N}$ and expresses certain partial sums of coefficients in terms of expressions involving roots of unity. By specializing $N$ to different values, we see that these expressions simplify in some cases and we obtain several nice identities involving these coefficients. We also use a result of Sudler (Quarterly J. Math. 1964) to obtain an asymptotic formula for the maximum absolute value of these coefficients.    
\end{abstract}
\maketitle
\section{Introduction}
In the study of partitions, one often comes across various $q$-series and $q$-products. In these kind of problems we may (a) either need to understand the possible product representation of a given $q$-series; (b) or need to understand the coefficients in the $q$-series expansion of a given $q$-product. Our purpose in this paper is the latter theme. Before embarking on the specific product we are interested in, we begin with a very well-known infinite $q$-product which is closely related to the product we are interested in:
\begin{equation}\label{qprod}
T_{s,\infty}(q):=\prod_{j=1}^\infty(1-q^j)^s,    
\end{equation}
where $s\in \mathbb{N}$ and which converges for $|q|<1$. We note here that the product in (\ref{qprod}) is very general in terms of $s$. There are only a handful of values of $s$ for which the infinite product in (\ref{qprod}) yields a series expansion with explicit coefficients. For $s=1$, the product in (\ref{qprod}) admits the following well-known (single) series expansion.
\begin{theorem}[Euler Pentagonal Number Theorem]\label{Eu}
We have
\begin{eqnarray*}
\prod_{j=1}^\infty(1-q^j)=\sum_{k=-\infty}^\infty (-1)^k q^{k(3k-1)/2}.
\end{eqnarray*}
\end{theorem}
Combinatorially, the coefficient of $q^n$ in the left-hand side of the identity in Theorem \ref{Eu} is the difference of the number of partitions of $n$ into an even number of parts and the number of partitions of $n$ into an odd number of parts. Establishing an involution, Franklin \cite{Fra} showed that this difference is zero unless $n=k(3k\pm 1)/2$, a pentagonal number, in which case the coefficient is $(-1)^k$, thereby yielding the right-hand side of the theorem. In \cite{And1}, Andrews gave a modern exposition of Euler's original proof of the theorem and also discussed some consequences. The case $s=3$ yields the following well-known series expansion of Jacobi:
\begin{theorem}[Jacobi]\label{Jac}
We have
\begin{eqnarray*}
\prod_{j=1}^\infty(1-q^j)^3=\sum_{k=0}^\infty (-1)^k(2k+1) \;q^{k(k-1)/2}.
\end{eqnarray*}
\end{theorem}
A Franklin-type involutive proof of Theorem \ref{Jac} was obtained by Joichi and Stanton \cite{JS}. For all other choices of $s$, there is no $q$-series expansion for the infinite product in (\ref{qprod}) with explicit coefficients. However in some cases, the $q$-product in \eqref{qprod} can be expressed as a double series representation of the Hecke-Rogers type (see \cite{And2}). These representations were first obtained by Rogers \cite{Rog} but systematically studied much later by Hecke \cite{Hec}. For instance, when $s=2$, the product \eqref{qprod} admits the following double series representation.   
\begin{theorem}[Hecke, Rogers]\label{s2}
We have
\begin{eqnarray*}
\prod_{j=1}^\infty (1-q^j)^2=\sum_{n=0}^\infty \sum_{-n/2\leq m\leq n/2}(-1)^{n+m}q^{(n^2-3m^2)/2+(n+m)/2}.
\end{eqnarray*}
\end{theorem}
Kac and Peterson \cite{KP} also obtained the identity in Theorem \ref{s2} in connection to character formulas for infinite dimensional Lie algebras and string functions. Bressoud \cite{Br} and Andrews \cite{And2} also obtained this identity. In the case $s=24$, the coefficients  of the infinite product in (\ref{qprod}) are closely related to the Ramanujan $\tau$ function. The Ramanujan $\tau$-function can be defined by 
\begin{eqnarray} \label{tau}
\sum_{n=0}^\infty \tau(n)q^n:=q\prod_{j=1}^\infty (1-q^j)^{24}
\end{eqnarray}
The $\tau$-function possess very nice arithmetic properties, see \cite{Ran}. In particular, the $\tau$-function is multiplicative, as originally observed by Ramanujan and later proved by Mordell \cite{Mor}. Lehmer \cite{Leh} conjectured that the $\tau$-function never vanishes which is still open. Dyson \cite{Dy} obtained a nice formula for $\tau(n)$ which has strong combinatorial flavour but not easily seen to be related with a direct counting of special combinatorial objects.

Consider the truncation of the $q$-product in (\ref{qprod}), that is, consider the polynomial
\begin{eqnarray}
T_{s,n}(q):=\prod_{j=1}^n(1-q^j)^s. 
\end{eqnarray}
It is clear that $N_{s,n}:=\;$deg$\;T_{s,n}=n(n+1)s/2$. 
%Let $p_{s,n}>N_{s,n}$ be any odd prime. 
Define the coefficients $t_{i,s,n}$ by
\begin{eqnarray}\label{Tpq}
T_{s,n}(q)&:=&\sum\limits_{i=0}^{N_{s,n}}t_{i,s,n}\;q^{i}.
\end{eqnarray}
$T_{s,n}(q)$ comes up quite naturally in the study of restricted partitions. Although $T_{s,n}(q)$ is the truncation of $T_{s,\infty}(q)$, very little is known about the coefficients in the $q$-series expansion of $T_{s,n}(q)$. For $s=1$, let $\mathcal{M}_n:=\max_{j}|t_{j,1,n}|$. Then Sudler \cite{Sud} showed that
\begin{eqnarray}\label{Sud4}
\log \mathcal{M}_n = Kn+O(\log n),
\end{eqnarray}
for some constant $K\approx 0.19861$. In a subsequent paper, Sudler \cite{Sud1} showed, by establishing an algebraic identity that the coefficients $t_{j,1,n}$ satisfies
\begin{eqnarray}\label{Sud3}
t_{j,1,n}=[q^{j-2n-2}](q^3;q)_\infty + O(1),
\end{eqnarray}
where $2n+2\leq j\leq 3n+2$. By establishing another algebraic identity and Dirichlet's box principle, he showed that $(q^3;q)_\infty$ has unbounded coefficients. This enabled him to show using (\ref{Sud3}) that $t_{j,1,n}$ are unbounded as $n+1>j-2n-2\rightarrow\infty$ as $n\rightarrow\infty$.  In \cite{Wri}, Wright improved the asymptotic estimate in (\ref{Sud4}). By obtaining a new formula for the truncation of the series in Theorem \ref{Eu}, Andrews and Merca \cite{AM} recently obtained infinitely many inequalities for $p(n)$, the unrestricted partition of $n$. For non-negative integers $m, r$, let ${m\brack r}$ denote the $q$-binomial coefficient defined by
\begin{eqnarray}
{m\brack r}=\left\{\begin{array}{cc}
    \dfrac{(1-q^m)(1-q^{m-1})\cdots (1-q^{m-r+1})}{(1-q)(1-q^2)\cdots (1-q^r)}, & \text{if}\;r\leq m, \\
    \mbox{}\\
     0,&\text{otherwise}. 
\end{array}\right.
\end{eqnarray}
Using Cauchy's theorem \cite[Theorem 2.1, pp. 17]{AndB}, it follows for $s=1$ that
\begin{eqnarray}\label{cauchy}
\prod_{j=1}^n (1-q^j)=\sum_{k=0}^n {n\brack k}(-1)^k q^{k(k+1)/2}.
\end{eqnarray}
Since each of ${n\brack k}$ is a polynomial in $q$ of degree $k(n-k)$ with integer coefficients \cite[Theorem 3.2, pp. 35]{AndB}, the right-hand side of (\ref{cauchy}) does not give us an explicit form of the coefficient (independent of $q$) of $q^m$. 

To this end, we adopt the following conventions. Given a polynomial $f(x)$, by $[x^{j}]f(x)$, we denote the coefficient of $x^{j}$ in $f(x)$. Let $d\in\mathbb{Z}$ be such that $0\leq d\leq a-1$. In what follows, let $S_{a,d}$ denote the arithmetic progression
\begin{eqnarray}
S_{a,d}:=\left\{am+d: m\in\mathbb{Z}\right\}.
\end{eqnarray}
The aim of this article is to obtain several results concerning the coefficients $t_{i,s,n}$. Let $N\in\mathbb{N}$ and consider the following sum:
\begin{eqnarray}\label{sumprog}
\sum_{\substack{0\leq i\leq N_{s,n}\\i\in S_{N,d}}}t_{i,s,n}.
\end{eqnarray}
where $0\leq d< N$. We obtain a nice formula involving roots of unity for the sum in (\ref{sumprog}). Under two suitable choices of $N$ (see Section \ref{MR}) we obtain two interesting results. In one case, the sum reduces to a single coefficient of $T_{s,n}(q)$, thereby yielding a formula for the individual coefficients involving roots of unity. In the other case, the sum involving roots of unity simplify and we obtain a nice closed expression for the sum of coefficients. Finally, we use a result due to Sudler \cite{Sud} to obtain an asymptotic estimate for the maximum absolute coefficients of $T_{s,n}(q)$. 

This paper is organized as follows. In Section \ref{NCB} we introduce a few notations, conventions and do some basic counting. In Section \ref{MR} we state our main results. In Section \ref{LWs} we recall Li and Wan's \cite{Li-Wan} sieving principle and also establish a few basic results. Finally in Section \ref{Ps} we obtain the proofs of our main results.
\section{Acknowledgement}
The research of the first author was supported by grant SFB F50-06 of the Austrian Science Fund (FWF). The authors thank George Andrews and Peter Paule for their feedback.
%The authors would like to thank George Andrews, Krishna Alladi and Peter Paule for their support and encouragement. 
\section{Notation, Conventions and Basic Counting}\label{NCB}
Let $n\in\mathbb{N}$. Consider the set $D_n=\{1,2, \ldots, n\}$. We define the following: 
\begin{eqnarray*}
\C_{e,s}(j,n)&:=&\#\left\{\varprod_{i=1}^s V_i\subset D_n^s:\sum_{i=1}^s|V_i|\equiv 0 \;(\mbox{mod}\;2), \sum_{i=1}^s\sum\limits_{x_{v_i}\in V_i} x_{v_i}= j\right\},\notag\\
\C_{o,s}(j,n)&:=&\#\left\{\varprod_{i=1}^s V_i \subset D_n^s:\sum_{i=1}^s|V_i|\equiv 1\;(\mbox{mod}\;2), \sum_{i=1}^s\sum\limits_{x_{v_i}\in V_i} x_{v_i}= j\right\}.
\end{eqnarray*}
It is now apparent that
\begin{equation}\label{co}
 t_{j,s,n}=\C_{e,s}(j,n)-\C_{o,s}(j,n).
\end{equation}
Combinatorially, $\C_{e,p}(j,n)$ (resp. $\C_{o,p}(j,n)$) counts the number of partitions of $j$ into an even (resp. odd) number of parts $\leq n$ where every part can repeat at most $s$ times. 

Let $N\in\mathbb{N}$ to be chosen appropriately later. Let $G=\mathbb{Z}_{N}$. %We note that $G\setminus D_{n}=\{0\}\cup\{\}$ is the trivial subgroup of $G$. 
Given $0\leq k_1, k_2,\ldots,k_s \leq |D_{n}|$ and $0\leq j < N$, define
\begin{eqnarray*}
M_{s,n,N}(k_1,k_2,\ldots,k_s,j):=\#\left\{\varprod_{i=1}^sV_i \subset D_n^s:|V_1|=k_1,\ldots,|V_s|=k_s, \sum_{i=1}^s\sum\limits_{x_{v_i}\in V_i} x_{v_i}\equiv j \pmod{N} \right\}, 
\end{eqnarray*}
and set
\begin{eqnarray*}
M_{s,n,N}(b)=\sum_{0\leq k_1,k_2,\ldots,k_s\leq |D_{n}|}(-1)^{k_1+k_2+\ldots+k_s}M_{s,n,N}(k_1,k_2,\ldots,k_s,b).
\end{eqnarray*}
From (\ref{co}), we see that
\begin{equation}\label{ps}
M_{s,n,N}(j)=\sum_{\substack{0\leq i\leq N_{s,n}\\i\in S_{N,j}}}t_{i,s,n}.
\end{equation}
We next introduce a few more notations. Let $(x)_k:=(x-1)(x-2)\ldots(x-k+1)$ denote the falling factorial. Let $\hat{G}$ be the set of complex-valued linear characters of $G$. By $\psi_{0}$, we denote the trivial character in $\hat{G}$. Let $X_{n,k}=D_n^k$ and $\overline{X}_{n,k}$ denote the subset of all tuples in $D_n^{k}$ with distinct coordinates. 
\section{Main results}\label{MR}
Our main results are below. 
\begin{theorem}\label{main00}
Let $j\in \mathbb{Z}_{N}$. Then we have 
\begin{eqnarray*}
M_{s,n,N}(j)=\sum_{\substack{0\leq i\leq N_{s,n}\\i\in S_{N,j}}} [q^{i}]T_{n,s}(q)=\dfrac{1}{N}\sum_{\psi\neq \psi_0}\psi^{-1}(j)\prod_{a\in D_n}(1-\psi(-a))^s
\end{eqnarray*}
where $\psi_0\neq \psi\in\hat{\mathbb{Z}}_{N}$.
\end{theorem}
\begin{theorem}\label{main0000}
Let $j\in \mathbb{Z}_{N}$. Then we have 
\begin{eqnarray*}
 M_{s,n,N}(j)=\left\{\begin{array}{cc}
 \dfrac{(2i)^{sn+1}}{N}\displaystyle\sum_{1\leq r\leq N/2}\sin\left(\pi r\cdot\dfrac{2j-sn(n+1)/2}{N}\right)\displaystyle\prod_{a\in D_n}\sin^s\left(\dfrac{\pi a r}{N}\right),& \emph{if}\;sn\equiv 1\;(\emph{mod}\;2)\\
 \dfrac{2(2i)^{sn}}{N}\displaystyle\sum_{1\leq r\leq N/2}\cos\left(\pi r\cdot\dfrac{2j-sn(n+1)/2}{N}\right)\displaystyle\prod_{a\in D_n}\sin^s\left(\dfrac{\pi a r}{N}\right),&\emph{otherwise}.
 \end{array}\right.
 \end{eqnarray*}
\end{theorem}
\begin{theorem}\label{main000}
Let $sn\equiv 1\;(\emph{mod}\;2)$. For a given $j\in\mathbb{Z}_{N}$, let $N-1\leq N_{s,n}$ be such that $2j\equiv N_{s,n}\;(\emph{mod}\;N)$. Then
\begin{eqnarray*}
M_{s,n,N}(j)=\sum_{\substack{0\leq i\leq N_{s,n}\\i\in S_{N, j}}} t_{i,s,n}=0.
\end{eqnarray*}
In particular, let $N'\mid N_{s,n}$. Then 
\begin{eqnarray*}
M_{s,n,N'}(0)=\sum_{0\leq \ell\leq N_{s,n}/N'}t_{\ell N',s,n}=0.
\end{eqnarray*}
\end{theorem}
\begin{theorem}\label{main0}
Let $j\in \mathbb{Z}_{N_{s,n}+1}$. Then we have 
\begin{eqnarray*}
t_{j,s,n}=\dfrac{1}{N_{s,n}+1}\sum_{\psi\neq \psi_0}\psi^{-1}(j)\prod_{a\in D_n}(1-\psi(-a))^s
\end{eqnarray*}
where $\psi_0\neq \psi\in\hat{\mathbb{Z}}_{N_{s,n}+1}$.
\end{theorem}
\begin{theorem}\label{main1}
%With $M_{s,n}(b)$ defined as in \emph{(\ref{ps})} and 
Let $j\in\mathbb{Z}_{n+1}$ we have
\begin{eqnarray*}
M_{s,n,n}(j)=\sum_{\substack{0\leq i\leq N_{s,n}\\i\in S_{n+1,j}}} t_{i,s,n}=\left\{\begin{array}{cc}
(n+1)^{s-1}\varphi(n+1),&\emph{if}\;j=0,\\
\mbox{}\\
-(n+1)^{s-1},&\emph{otherwise.}
\end{array}\right.
\end{eqnarray*}
\end{theorem}
\begin{cor}\label{main00cor}
Let $1\leq N\leq n-1$. Then for any $j\in \mathbb{Z}_{N}$ we have
\begin{eqnarray*}
M_{s,n}(j)=\sum_{\substack{0\leq i\leq N_{s,n}\\i\in S_{N,j}}}t_{i,s,n}=0.
\end{eqnarray*}
\end{cor}
\begin{cor}\label{div1}
For positive odd integers $s$ and $n$, let $D$ be any divisor of $N_{s,n}$ such that $N_{s,n}/2<D\leq N_{s,n}$. Then
\begin{eqnarray*}
t_{D,s,n}=-1,\hspace{1cm}\;t_{(N_{s,n}-D),s,n}=1.
\end{eqnarray*}
\end{cor}
\begin{cor}\label{peak1}
Let $n\equiv 3\;\pmod 4$ and $s$ be an odd integer. Then 
\begin{eqnarray*}
t_{\frac{sn(n+1)}{4},s,n}=0.
\end{eqnarray*}
\end{cor}
In view of \eqref{tau}, let us denote by $\tau_{n}(j)$ the $jth$ coefficient of the following truncated product:
\begin{eqnarray*}
\prod_{k=1}^n(1-q^k)^{24}:=\sum_{j=0}^{12n(n+1)}\tau_{n}(j)\;q^j.
\end{eqnarray*}
Then
\begin{cor}
Let $j\in\mathbb{Z}_{n+1}$. We have
\begin{eqnarray*}
\sum_{\ell=0}^{(12n(n+1)-j)/(n+1)}\tau_{n}(\ell(n+1)+j)=\left\{\begin{array}{cc}
(n+1)^{23}\varphi(n+1),&\emph{if}\;j=0,\\
\mbox{}\\
-(n+1)^{23},&\emph{otherwise.}
\end{array}\right.
\end{eqnarray*}
\end{cor}
Let $\mathcal{M}_{n,s}=\max_{j} |t_{j,s,n}|$. Then using a result due to Sudler \cite{Sud} we prove the following:
\begin{theorem}\label{maxpeak}
For a sufficiently large $n$ we have
\begin{eqnarray*}
\mathcal{M}_{n,s}=e^{sKn+O(\log n)},
\end{eqnarray*}
where $K=\log 2 + \max\left(\displaystyle w^{-1}\int_{0}^w\log \sin \pi t\;dt\right)\approx 0.19861$, with $\frac{1}{2}<w<1$.
\end{theorem}
\section{Li-Wan Sieve}\label{LWs}
The quantity $M_{s,n,N}(k_1,k_2,\ldots,k_s,j)$ is the number of certain type of subsets of $D_n^s$. We next apply some elementary character theory to estimate it. 

We note that 
\begin{eqnarray*}
\rho:=\sum\limits_{\psi \in \hat{G}}\psi
\end{eqnarray*}
is the regular character of $G$. It is well-known that $\rho(g)=0$ for all $g\in G\setminus \{0\}$, and that $\rho(0)=|G|=N$.
Given $0< r \leq |D_{n}|$, a character $\psi \in \hat{G}$, and $\bar{x}=(x_{1},\ldots ,\ x_{r})$, we set
\begin{eqnarray*}
\prod_{i=1}^r\psi(x_{i}):=f_{\psi}(\bar{x}),\ \text{and}\hspace{1cm}
\mathcal{S}(\bar{x}):=\sum_{i=1}^r x_i.
\end{eqnarray*}
Let $Y_{n,s}^{k_1,k_2,\ldots,k_s}$ denote the cartesian product $\prod_{i=1}^s \overline{X}_{n,k_i}$. Then we have
\begin{eqnarray}
k_1!\ldots k_s!M_{s,n,N}(k_1,k_2,\ldots,k_s,j)&=&\dfrac{1}{N}\sum\limits_{(\bar{x}_1,\bar{x}_2,\ldots,\bar{x}_s)\in Y_{n,s}^{k_1,k_2,\ldots,k_s}} \sum \limits_{\psi \in \hat{G}} \psi(\mathcal{S}(\bar{x}_1)+\mathcal{S}(\bar{x}_2)+\ldots+\mathcal{S}(\bar{x}_s)-j)\notag\\
&=&\dfrac{1}{N}\left(\prod_{i=1}^s\left(n \right)_{k_{i}}+\sum\limits_{(\bar{x}_1,\bar{x}_2,\ldots,\bar{x}_s)\in Y_{p,s}^{k_1,k_2,\ldots,k_s}} \sum \limits_{\psi_0\neq \psi \in \hat{G}}\psi^{-1}(j)\prod_{i=1}^s\psi(\mathcal{S}(\bar{x}_i))\right).\notag
\end{eqnarray}
In the right-hand side above we interchange the sums to get,
\begin{eqnarray}
k_1!\ldots k_s!M_{s,n,N}(k_1,k_2,\ldots,k_s,j)&=&\dfrac{1}{N}\left(\prod_{i=1}^s\left(n \right)_{k_{i}}\notag+\;\sum \limits_{\psi_0\neq \psi \in \hat{G}}\psi^{-1}(j)\sum\limits_{(\bar{x}_1,\bar{x}_2,\ldots,\bar{x}_s)\in Y_{n,s}^{k_1,k_2,\ldots,k_s}} \prod_{i=1}^sf_\psi(\bar{x}_i)\right)\notag\\
&=&\dfrac{1}{N}\left(\prod_{i=1}^s\left(n \right)_{k_{i}}+\sum \limits_{\psi_0\neq \psi \in \hat{G}}\psi^{-1}(j)\prod_{i=1}^s\left(\sum\limits_{\bar{x}_i\in \overline{X}_{n,k_i}} f_\psi(\bar{x}_i)\right)\right).
\end{eqnarray}
For a $Y \subset X_{n,k}$ and a character $\psi \in \hat{G}$, set $F_{\psi}(Y):=\sum\limits_{\bar{y} \in Y} f_{\psi}(\bar{y})$. We now have 
\begin{eqnarray}\label{rc}
k_1!k_2!\ldots k_s!M_{s,n}(k_1,k_2,\ldots,k_s,b)&=&\dfrac{1}{N}\prod_{i=1}^s\left(n \right)_{k_{i}}+\dfrac{1}{N}\sum \limits_{\psi_0\neq \psi \in \hat{G}}\psi^{-1}(j)\prod_{i=1}^sF_{\psi}(\overline{X}_{n,k_i})
\end{eqnarray}
We now estimate sums of the form $F_{\psi}(\overline{X}_{n,k})$.
The symmetric group $S_{k}$ acts naturally on $X_{n,k}=D_n^{k}$. Let $\tau \in S_{k}$ be a permutation whose cycle decomposition is
\begin{eqnarray*}
\tau=(i_1i_2\ldots i_{a_1})(j_1j_2\ldots j_{a_2})\ldots (\ell_1\ell_2\ldots \ell_{a_s})
\end{eqnarray*}
where $a_i\geq 1, 1\leq i\leq s$. We define 
\begin{eqnarray*}
X_{n,k}^{\tau}:=\left\{(x_1,x_2,\ldots,x_k)\in X_{n,k}: x_{i_1}=\ldots=x_{i_{a_1}}, \ldots,x_{\ell_1}=\ldots=x_{\ell_{a_s}} \right\}.
\end{eqnarray*}
In other words, $X_{n,k}^{\tau}$ is the set of elements in $X_{n,k}$ fixed under the action of $\tau$. Let $C_{k}$ be a set of conjugacy class representatives of $S_{k}$.  
%Define $F_{\tau}(\psi)=\sum\limits_{y\in X_{\tau}}f_{\psi}(y)$. 
Let us denote by $C(\tau)$ the number of elements conjugate to $\tau$. 
Now for any $\tau \in S_{k}$, we have  $\tau(X_{n,k})=X_{n,k}$. We note that for any pair $\tau$, $\tau'$ of conjugate permutations, and for any $\psi \in \hat{G}$, we have $F_{\psi}(\overline{X}^\tau_{n,k})=F_{\psi}(\overline{X}^{\tau'}_{n,k_i})$. That is, according to the definitions in \cite{Li-Wan}, $X_{n,k}$ is symmetric and $f_\psi$ is normal on $X$. Thus we have the following result which is essentially \cite[Proposition 2.8]{Li-Wan}.
\begin{prop}\label{LWS}
We have
\begin{equation*}
F_{\psi}(\overline{X}_{n,k})=\sum\limits_{\tau \in C_{k}} \sign(\tau)C(\tau)F_{\psi}(\overline{X}^\tau_{n,k}).
\end{equation*}
\end{prop}
\subsection{Some useful lemmas}
The following lemma exhibits the relationship between $F_{\psi}(\overline{X}^\tau_{n,k})$ and the cycle structure of $\tau$.  
\begin{lem}\label{Ftau}
Let $\tau \in C_{k}$ be the representative whose cyclic structure is associated with the partition $(1^{c_{1}},2^{c_{2}},\ldots k^{c_{k}})$ of $k$. Then we have $F_{\psi}(X^{\tau}_{n,k})=\prod_{i=1}^{k}(\sum\limits_{a \in D_n}\psi^{i}(a))^{c_{i}}$.
\end{lem} 
\begin{proof}
Recall that
\begin{align*}
F_{\psi}(X^{\tau}_{n,k})&=\sum\limits_{\bar{x} \in X_{n,k}^{\tau}} \prod_{i=1}^{k}\psi(x_{i})\\
&=\sum\limits_{\bar{x}\in X_{n,k}^{\tau}} \prod_{i=1}^{c_{1}}\psi(x_{i})\prod_{i=1}^{c_{2}}\psi^{2}(x_{c_{1}+2i})\ldots \prod_{i=1}^{c_{k}}\psi^{k}(x_{c_{1}+c_{2}\ldots+ki})\\
&= \prod_{i=1}^{k}(\sum\limits_{a \in D_n}\psi^{i}(a))^{c_{i}}.
\end{align*}
\end{proof}
Given $\chi\in \hat{G}$ define
\begin{eqnarray}\label{schi}
s_{D_n}(\chi):=\sum\limits_{a\in D_n} \chi(a).
\end{eqnarray}
Let $N(c_{1},c_{2},\ldots c_{k})$ denote the number of elements of $S_{k}$ of cycle type $(c_{1},c_{2},\ldots c_{k})$. It is well-known (see, for example, \cite{Stan}) that
\begin{eqnarray}\label{Ncom}
N(c_{1},c_{2},\ldots c_{k})=\dfrac{k!}{1^{c_1}c_1!2^{c_2}c_2!\ldots k^{c_k}c_k!}.
\end{eqnarray}
Then
\begin{lem}\label{Hpsi}
We have
\begin{eqnarray*}
F_{\psi}(\overline{X}_{n,k})=(-1)^k\sum_{\sum_{i}ic_i=k}N(c_1,c_2,\ldots, c_k)\prod_{i=1}^{k}(-s_{D_n}(\psi^i))^{c_{i}}.
\end{eqnarray*}
\end{lem}
\begin{proof}
To prove this lemma, we first note that $\sign(\tau)=(-1)^{k-\sum_{i}c_i}$. Also the cyclic structure for every $\tau\in C_k$ can be associated to a partition of $k$ of the form $(1^{c_1}, 2^{c_2},\ldots,k^{c_k})$. Hence the right-hand sum in Proposition \ref{LWS} runs over all such partitions of $k$. Noting that the conjugate permutations have same cycle type, and there are exactly $N(c_1,c_2,\ldots,c_k)$ permutations with cycle type $(c_{1},c_{2},\ldots c_{k})$ we conclude, in view of Lemma \ref{Ftau} that
\begin{eqnarray*}
F_{\psi}(\overline{X}_{n,k})=(-1)^k\sum_{\sum_{i}ic_i=k}N(c_1,c_2,\ldots, c_k)\prod_{i=1}^{k}(-\sum\limits_{a \in D_n}\psi^{i}(a))^{c_{i}}.
\end{eqnarray*}
\end{proof}
Next, we define the following polynomial in $k$ variables:
\begin{eqnarray}\label{Zgen}
Z_{k}(t_{1},\ldots, t_{k}):= \sum\limits_{\sum ic_{i}=k}N(c_{1},\ldots,c_{k})t_{1}^{c_{1}}\ldots t_{k}^{c_{k}}.
\end{eqnarray}
From Lemma \ref{Hpsi} and \eqref{Zgen} we immediately see that
\begin{cor}\label{HZ}
We have
\begin{eqnarray*}
F_{\psi}(\overline{X}_{n,k})=(-1)^kZ_k(-s_{D_n}(\psi),-s_{D_n}(\psi^2),\ldots,-s_{D_n}(\psi^k))
\end{eqnarray*}
where for $\chi\in\hat{G}$, $s_{D_n}(\chi)$ is as in \emph{(\ref{schi})}.
\end{cor}
\subsection{Some combinatorial functions and estimates}
We now evaluate $Z_{k}(\delta^{\psi}(1),\ldots, \delta^{\psi}(k))$.
% Observing $N(c_{1},c_{2},\ldots c_{k})=\dfrac{k!}{(1^{c_{1}})(c_{1}!)(2^{c_{2}})(c_{2}!)\ldots (k^{c_{k}}) (c_{k}!) }$, we see that
From (\ref{Ncom}) and (\ref{Zgen}) we immediately deduce the following:
\begin{lem}[Exponential generating function]\label{Egf}
We have
\begin{eqnarray*}
\sum\limits_{k \geq 0}Z_{k}(t_{1},t_{2}, \ldots t_{k})\dfrac{u^{k}}{k!}=e^{ut_{1}+u^{2}\frac{t_{2}}{2}+\ldots }
\end{eqnarray*}
\end{lem}
\begin{lem}\label{speZ1}
Let $D\subseteq \mathbb{N}$. Then
\begin{eqnarray*}
Z_k\left(-s_{D}(\psi), -s_{D}(\psi^2),\ldots,-s_D(\psi^k)\right)=\left[\dfrac{u^k}{k!}\right]\prod_{a\in D}(u-\psi(-a)).
\end{eqnarray*}
\end{lem}
More generally, one can prove the following stronger result.
\begin{lem}\label{speZ}
Let $t_j=b_{j(\emph{mod}\;d)}$ iff $d\nmid j$ and $t_j=a$ iff $d\mid j$, then
\begin{eqnarray*}
Z_k(b_1, b_2, \ldots, b_{d-1},a,b_1,b_2,\ldots, b_{d-1}, a, b_1, b_2, \ldots,b_{d-1}, a, \ldots)=\left[\dfrac{u^k}{k!}\right]\dfrac{1}{(1-u^{d})^{a/d}\displaystyle\prod_{r=1}^d(1-ue^{-2i\pi r/d})^{\omega(r,d)/d}},
\end{eqnarray*}
where 
\begin{eqnarray*}
\omega(r,d):=\sum_{j=1}^{d-1}b_je^{2i\pi rj/d}.
\end{eqnarray*}
\end{lem}
We only prove Lemma \ref{speZ1}. Using Lagrange's interpolation one can prove Lemma \ref{speZ}.
\begin{proof}[Proof of Lemma \emph{\ref{speZ1}}]
Let $d=o(\psi)$. From Lemma \ref{Zgen} we have
\begin{align*}%\label{genfunev}
\sum\limits_{k \geq 0}Z_k\left(-s_{D}(\psi), -s_{D}(\psi^2),\ldots,-s_D(\psi^k)\right)\dfrac{u^{k}}{k!}
&=\exp\left\{\sum_{j=1}^{d}-s_D(\psi^j)\sum_{\ell=0}^\infty \dfrac{u^{d\ell +j}}{d\ell+j}\right\}\notag\\
&=\exp\left\{-\sum_{j=1}^{d}\sum_{a\in D}\psi^j(a)\sum_{\ell=0}^\infty \dfrac{u^{d\ell +j}}{d\ell+j}\right\}\notag\\
&=\exp\left\{-\sum_{j=1}^{d}\sum_{a\in D}\psi^j(a)\int_0^u\dfrac{x^{j-1}}{1-x^d}dx\right\}\notag\\
&=\exp\left\{-\sum_{j=1}^{d}\sum_{a\in D}\int_0^u\dfrac{\psi^j(a)\cdot x^{j-1}}{1-x^d}dx\right\}\notag\\
&=\exp\left\{-\sum_{a\in D}\int_{0}^u\dfrac{\sum_{j=1}^{d}\psi^j(a)\cdot x^{j-1}}{1-x^d}dx\right\}\notag\\ 
&=\exp\left\{-\sum_{a\in D}\int_{0}^u\dfrac{\psi(a)(1-\psi^d(a)x^d)}{(1-\psi(a)x)(1-x^d)}\right\}\notag\\
&=\exp\left\{\sum_{a\in D}\int_{0}^u\dfrac{dx}{(x-\psi(-a))}\right\}\notag\\
&=\exp\left\{\sum_{a\in D}\log(u-\psi(-a))\right\}\notag\\
&=\prod_{a\in D}(u-\psi(-a))
\end{align*}
which proves the lemma.
\end{proof}
From Corollary \ref{HZ} and Lemma \ref{speZ1} we obtain:
\begin{lem}\label{Hfines}
We have
\begin{eqnarray*}
F_{\psi}(\overline{X}_{n,k})=(-1)^k\left[\dfrac{u^k}{k!}\right]\prod_{a\in D_n}(u-\psi(-a)).
\end{eqnarray*}
\end{lem}
\section{Proofs of the main results}\label{Ps}
\begin{proof}[Proof of Theorem \emph{\ref{main00}}]
From (\ref{rc}) we have 
\begin{eqnarray}\label{Mkb}
%M_{p,n}(k,b)=N^{-1}\left\{\binom{(p-1)N/p}{k}+Q_{k}+R_{k}+S_{k}\right\}
M_{s,n,N}(k_1,k_2,\ldots,k_s,j)&=&\dfrac{1}{N}\left\{\prod_{i=1}^s\binom{n}{k_i}+P_{N,k_1,k_2,\ldots,k_s,}\right\}
\end{eqnarray}
where
\begin{align}\label{I-II-III}
P_{N,k_1,k_2,\ldots,k_s}= \dfrac{1}{k_1!k_2!\ldots k_s!}\sum\limits_{\psi\neq \psi_0}\psi^{-1}(b)\prod_{i=1}^sF_{\psi}(\overline{X}_{n,k_i})
\end{align}
Using Lemma \ref{Hfines}, we see that
\begin{eqnarray}\label{I-II-III-r}
P_{N,k_1,k_2,\ldots,k_s}=\dfrac{(-1)^{k_1+k_2+\ldots+k_s}}{k_1!k_2!\ldots k_s!}\sum\limits_{\psi\neq \psi_0}\psi^{-1}(b) \prod_{i=1}^s\left[\dfrac{u^{k_i}}{k_i!}\right]\prod_{a\in D_n}(u-\psi(-a)).
\end{eqnarray}
Recall that 
\begin{eqnarray}\label{Mb}
\ \ \ \ \ \ \ \ \ \ M_{s,n,N}(j)=\sum_{0\leq k_1,k_2,\ldots,k_s\leq |D_{n}|}(-1)^{k_1+k_2+\ldots+k_s}M_{s,n,N}(k_1,k_2,\ldots,k_s,j).
\end{eqnarray}
Using the well-known fact
\begin{eqnarray*}
\sum\limits_{k=0}^{|D_n|} (-1)^{k}\binom{|D_n|}{k}=0,
\end{eqnarray*}
we see that 
\begin{eqnarray}\label{pbino}
\mbox{}\\
\sum_{0\leq k_1,\ldots,k_s\leq |D_n|}(-1)^{k_1+k_2+\ldots+k_s}\prod_{i=1}^s\binom{n}{k_i}=\left(\sum\limits_{k=0}^{|D_n|} (-1)^{k}\binom{|D_n|}{k}\right)^s=0.\notag
\end{eqnarray}
Thus (\ref{Mkb}), (\ref{Mb}) and (\ref{pbino}) yield,
\begin{eqnarray}\label{Mbr}
M_{s,n,N}(j)&=&\dfrac{1}{N}\sum_{0\leq k_1,k_2,\ldots,k_s\leq |D_{n}|}(-1)^{k_1+k_2+\ldots+k_s}P_{N,k_1,k_2,\ldots,k_s}.
\end{eqnarray}
From (\ref{I-II-III-r}) we have
\begin{eqnarray}\label{summulti}
&&\sum_{0\leq k_1,k_2,\ldots,k_s\leq |D_{n}|}(-1)^{k_1+k_2+\ldots+k_s}P_{N,k_1,k_2,\ldots,k_s}\notag\\&=&\sum\limits_{\psi\neq \psi_0}\psi^{-1}(j) \sum\limits_{0\leq k_1,\ldots,k_s\leq |D_n|}\dfrac{1}{k_1!\ldots k_s!} \prod_{i=1}^s\left[\dfrac{u^{k_i}}{k_i!}\right]\prod_{a\in D_n}(u-\psi(-a))\notag\\
&=&\sum\limits_{\psi\neq \psi_0}\psi^{-1}(j) \sum\limits_{0\leq k_1,k_2,\ldots,k_s\leq |D_n|}\prod_{i=1}^s\left[u^{k_i}\right]\prod_{a\in D_n}(u-\psi(-a))\notag\\
&=&\sum\limits_{\psi\neq \psi_0}\psi^{-1}(j)\left(\sum\limits_{k=0}^{|D_n|}[u^{k}]\prod_{a\in D_n}(u-\psi(-a))\right)^s.
\end{eqnarray}
It is clear that
\begin{eqnarray}
\sum\limits_{k=0}^{|D_n|}[u^{k}]\prod_{a\in D_n}(u-\psi(-a))=\prod_{a\in D_n}(1-\psi(-a)),
\end{eqnarray}
and thus (\ref{summulti}) implies 
\begin{eqnarray}\label{finpn}
\sum_{0\leq k_1,k_2,\ldots,k_s\leq |D_{n}|}(-1)^{k_1+k_2+\ldots+k_s}P_{N,k_1,k_2,\ldots,k_s}=\sum\limits_{\psi\neq \psi_0}\psi^{-1}(j)\prod_{a\in D_n}(1-\psi(-a))^s.
\end{eqnarray}
Finally, from (\ref{Mbr}) and (\ref{finpn}) we get
\begin{eqnarray}
M_{s,n,N}(j)=\dfrac{1}{N}\sum_{\psi\neq \psi_0}\psi^{-1}(j)\prod_{a\in D_n}(1-\psi(-a))^s,
\end{eqnarray}
which proves the result.
 \end{proof}
 \begin{proof}[Proof of Theorem \emph{\ref{main0000}}]
 We note that 
 \begin{eqnarray}\label{iden2}
 \sin\theta=\dfrac{e^{i\theta}-e^{-i\theta}}{2i},\hspace{1cm}\cos\theta=\dfrac{e^{i\theta}+e^{-i\theta}}{2}.
 \end{eqnarray}
 First, let us assume that $N$ is odd. Then for every $\psi_0\neq \psi\in\hat{\mathbb{Z}}_{N}$, there is an element in $\{-N/2\leq r\leq N/2\}$ such that $\psi(m)=e^{2i\pi mr/(N)}$.  
 
 Using the first identity in (\ref{iden2}) in the sum in the right-hand side of Theorem \ref{main00} we get
 \begin{eqnarray}\label{sincos1}
 M_{s,n,N}(j)&=&\dfrac{1}{N}\sideset{}{'}\sum_{-N/2\leq r\leq N/2}e^{2ij\pi r/N}e^{-i\pi rsn(n+1)/2N}\prod_{a\in D_n}\left(e^{i\pi ar/N}-e^{-i\pi ar/N}\right)^s\notag\\
 &=&\dfrac{(2i)^{sn}}{N}\sideset{}{'}\sum_{-N/2\leq r\leq N/2}e^{i\pi r(2j-sn(n+1)/2)/N}\prod_{a\in D_n}\sin^s\left(\dfrac{\pi a r}{N}\right),
 \end{eqnarray}
 where $\sideset{}{'}\sum$ indicates $r\neq 0$. If $N$ is even, $r=-N/2\equiv N/2\;(\mbox{mod}\;N)$ corresponds to the character of order $2$. Thus in this case, except for this character, for every other character, there is an element in $\{-N/2< r< N/2\}$ such that $\psi(m)=e^{2i\pi mr/N}$. However, we note that for $r=\pm N/2$, the product $\prod_{a\in D_n}\sin^s\left(\dfrac{\pi a r}{N}\right)=0$ since $2\in D_n$. Thus in this case too, we have
 \begin{eqnarray}\label{sincos2}
 M_{s,n,N}(j)=\dfrac{(2i)^{sn}}{N}\sideset{}{'}\sum_{-N/2\leq r\leq N/2}e^{i\pi r(2j-sn(n+1)/2)/N}\prod_{a\in D_n}\sin^s\left(\dfrac{\pi a r}{N}\right).
 \end{eqnarray}
 Next we combine terms in (\ref{sincos1}) and (\ref{sincos2}) corresponding to $r$ and $-r$ to get
 \begin{eqnarray}\label{cond23}
 \mbox{}\\
 &&M_{s,n,N}(j)=\dfrac{(2i)^{sn}}{N}\sum_{1\leq r\leq N/2}\left(e^{i\pi r(2j-sn(n+1)/2)/N}\prod_{a\in D_n}\sin^s\left(\dfrac{\pi a r}{N}\right)\right.\notag\\&&\left.+\;e^{-i\pi r(2j-sn(n+1)/2)/N}\prod_{a\in D_n}\sin^s\left(-\dfrac{\pi a r}{N}\right)\right)\notag\\
 &=&\dfrac{(2i)^{sn}}{N}\sum_{1\leq r\leq N/2}\left(e^{i\pi r(2j-sn(n+1)/2)/N}+(-1)^{sn}e^{-i\pi r(2j-sn(n+1)/2)/N}\right)\prod_{a\in D_n}\sin^s\left(\dfrac{\pi a r}{N}\right).\notag
 \end{eqnarray}
 Using the identities in (\ref{iden2}), (\ref{cond23}) yields
 \begin{eqnarray*}
 M_{s,n,N}(j)=\left\{\begin{array}{cc}
 \dfrac{(2i)^{sn+1}}{N}\displaystyle\sum_{1\leq r\leq N/2}\sin\left(\pi r\cdot\dfrac{2j-sn(n+1)/2}{N}\right)\displaystyle\prod_{a\in D_n}\sin^s\left(\dfrac{\pi a r}{N}\right),& \mbox{if}\;sn\equiv 1\;(\mbox{mod}\;2)\\
 \dfrac{2(2i)^{sn}}{N}\displaystyle\sum_{1\leq r\leq N/2}\cos\left(\pi r\cdot\dfrac{2j-sn(n+1)/2}{N}\right)\displaystyle\prod_{a\in D_n}\sin^s\left(\dfrac{\pi a r}{N}\right),&\mbox{otherwise}.
 \end{array}\right.
 \end{eqnarray*}
 \end{proof}
 \begin{proof}[Proof of Theorem \emph{\ref{main000}}]
 Let $j\in\mathbb{Z}_{N}$. If $N\leq N_{s,n}+1$ be such that $2j\equiv N_{s,n}\;(\mbox{mod}\;N)$, we find $\sin\left(\pi r\dfrac{2j-N_{s,n}}{N}\right)=0$ for all $1\leq r\leq N/2$. Since $sn\equiv 1\;(\mbox{mod}\;2)$, using Theorem \ref{main0000} we obtain
 \begin{eqnarray}
 M_{s,n}(j)&=&0.
 \end{eqnarray}
If $N'\mid N_{s,n}$ then taking $j=0$ above, we find
\begin{eqnarray}
M_{s,n}(0)=0
\end{eqnarray}
which proves our assertion.
 \end{proof}
 \begin{proof}[Proof of Theorem \emph{\ref{main0}}]
 To prove this theorem we only have to note that for $N=N_{s,n}+1$, $M_{s,n,N_{s,n}}(j)=t_{j,s,n}$. Now applying Theorem \ref{main00} we obtain the theorem.  
 \end{proof}
 \begin{proof}[Proof of Theorem \emph{\ref{main1}}]
 To prove this theorem we choose $N=n+1$. Thus $|\hat{G}|=n+1$. If $\psi\neq \psi_0$ be such that $o(\psi)<n+1$. Then there exists an $a\in D_n$ such that $o(\psi)=a$, in which case we have
 \begin{eqnarray}\label{zerocha}
 \prod_{a\in D_n}(1-\psi(-a))=0.
 \end{eqnarray}
 If $o(\psi)=n+1$ we have
 \begin{eqnarray}\label{nonzerocha}
 \prod_{a\in D_n}(1-\psi(-a))=\lim_{z\rightarrow 1}\dfrac{1-z^{n+1}}{1-z}=n+1.
 \end{eqnarray}
 In view of (\ref{zerocha}) and (\ref{nonzerocha}) we see from Theorem \ref{main00} that 
 \begin{eqnarray}
 M_{s,n,n}(j)&=&\dfrac{1}{n+1}\sum_{\substack{\psi\neq \psi_0\\o(\psi)=n+1}}\psi^{-1}(j)(n+1)^s\notag\\
 &=&(n+1)^{s-1}\sum_{\substack{\psi\neq \psi_0\\o(\psi)=n+1}}\psi^{-1}(j).
 \end{eqnarray}
 In view of the fact that there are exactly $\varphi(n+1)$ characters each of order $n+1$, we see that
 \begin{eqnarray}
 \sum_{\substack{\psi\neq \psi_0\\o(\psi)=n+1}}\psi^{-1}(j)=\left\{\begin{array}{cc}
 \varphi(n+1),& \text{if}\;j=0,\\
 -1,&\text{otherwise},
 \end{array}\right.
 \end{eqnarray}
 which proves the theorem.
 \end{proof}
 \begin{proof}[Proof of Corollary \emph{\ref{div1}}]
 Using Theorem \ref{main000}, we see that 
 \begin{eqnarray}\label{div2}
 \sum_{0\leq \ell\leq N_{s,n}/D}t_{\ell D,s,n}=0.
 \end{eqnarray}
 If $N_{s,n}/2<D\leq N_{s,n}$, the sum in (\ref{div2}) consists only of two terms corresponding to $\ell=0, 1$. Thus
 \begin{eqnarray*}
 1+t_{D,s,n}=0.
 \end{eqnarray*}
 Note that $T_{s,n}(q)$ is an anti-palindromic polynomial since 
 \begin{eqnarray}
 q^{N_{s,n}}T_{s,n}(1/q)=(-1)^{sn} T_{s,n}(q)=-T_{s,n}(q),
 \end{eqnarray}
 since $s$ and $n$ are both odd. Thus 
 \begin{eqnarray}
 t_{(N_{s,n}-D),s,n}=-t_{D,s,n}=1.
 \end{eqnarray}
 \end{proof}
 \begin{proof}[Proof of Corollary \emph{\ref{peak1}}]
 Since $n\equiv 3\pmod 4$, we can choose $N'=N_{s,n}/2$ in Theorem \ref{main000}. Thus we have
 \begin{eqnarray}\label{pair2}
 1+t_{N_{s,n}/2,s,n}+t_{N_{s,n},s,n}=0.
 \end{eqnarray}
 Note that 
 \begin{eqnarray}\label{antipalin}
 q^{N_{s,n}}T_{s,n}(1/q)=(-1)^{sn}T_{s,n}(q)=-T_{s,n}(q),
 \end{eqnarray}
 since $s$ and $n$ are both odd. Thus $t_{N_{s,n},s,n}=-t_{0,s,n}=-1$. Hence (\ref{pair2}) yields
 \begin{eqnarray*}
 t_{N_{s,n}/2,s,n}=0.
 \end{eqnarray*}
 \end{proof}
 \begin{proof}[Proof of Theorem \emph{\ref{maxpeak}}]
 Following Sudler \cite{Sud}, we obtain by Cauchy's formula that 
 \begin{eqnarray}\label{maxpeak1}
 |t_{j,s,n}|=\left|\dfrac{1}{2i\pi}\int_{|q|=1}\dfrac{T_{s,n}(q)}{q^{j+1}}dq\right|\leq \dfrac{1}{2\pi}\max_{|q|=1}|T_{s,n}(q)|\int_{|q|=1}\dfrac{|dq|}{|q|^{j+1}}=\max_{|q|=1}|T_{s,n}(q)|.
 \end{eqnarray}
 On the other hand we have
 \begin{eqnarray}\label{maxpeak2}
 \max_{|q|=1}|T_{s,n}(q)|\leq \sum_{0\leq j\leq N_{s,n}}|t_{j,s,n}|\leq (N_{s,n}+1)\max_{j}|t_{j,s,n}|. 
 \end{eqnarray}
 Hence (\ref{maxpeak1}) and (\ref{maxpeak2}) imply
 \begin{eqnarray}\label{Sud00}
 \log \max_{j}|t_{j,s,n}|=\log \max_{|q|=1}|T_{s,n}(q)|+O(\log n).
 \end{eqnarray}
 Thus the theorem follows if we show that 
 \begin{eqnarray}\label{Sud2}
 \log \max_{|q|=1}|T_{s,n}(q)|=sKn+O(\log n).
 \end{eqnarray}
 We note that
 \begin{eqnarray}\label{Sud0}
 \max_{|q|=1}|T_{s,n}(q)|=\max_{|q|=1}\left|\prod_{j=1}^n(1-q^j)\right|^s=\left(\max_{|q|=1}\left|\prod_{j=1}^n(1-q^j)\right|\right)^s.
 \end{eqnarray}
 From \cite{Sud}, we have 
 \begin{eqnarray}\label{Sud1}
 \log\max_{|q|=1}\left|\prod_{j=1}^n(1-q^j)\right|=Kn+O(\log n).
 \end{eqnarray}
 Hence in view of (\ref{Sud00}), (\ref{Sud0}) and (\ref{Sud1}), the estimate (\ref{Sud2}) and hence the theorem follow. 
 \end{proof}

% So we conclude that (a)if $p\mid b$, then $\left|M(b)- \frac{(p-1)p^{\frac{N}{p}}}{N}\right| \leq 2^{N/p}+p^{N/2p}$; and (b)if $p\nmid b$, then $\left|M(b)+ \frac{p^{\frac{N}{p}}}{N}\right| \leq 2^{N/p}+p^{N/2p}$
\end{document}